\documentclass[12pt,reqno]{siugrad51new}
\usepackage{amsmath, amssymb, amsthm, latexsym, mathtools,bold-extra}
\usepackage{adjustbox}
\makeatletter
\providecommand*{\dashv}{%
  \mathrel{%
    \mathpalette\@dashv\vdash
  }%
}
\newcommand*{\@dashv}[2]{%
  \reflectbox{$\m@th#1#2$}%
}
\usepackage{tikz-cd}
\usepackage[colorlinks=false,hyperfootnotes=false]{hyperref}
\renewcommand*\thesection{\arabic{section}}%%%%
\theoremstyle{plain}
\newtheorem{theorem}{Theorem}[section]

\newtheorem{corollary}[theorem]{Corollary}

\newtheorem{proposition}[theorem]{Proposition}
\theoremstyle{definition}
\newtheorem{example}[theorem]{Example}
\newtheorem{fact}[theorem]{Fact}
\newtheorem{definition}[theorem]{Definition}
\newtheorem{remark}[theorem]{Remark}

\newtheorem{ques}[theorem]{Question}

\newtheorem{conj}[theorem]{Conjecture}
\newtheorem*{theorem*}{Theorem}

\newcommand{\lk}{\leq_{\bf K}}

\newcommand{\mn}{\mathfrak{C}}
\newcommand{\oop}{\operatorname}
\newcommand{\gtp}{\oop{gtp}}
\newcommand{\gs}{\oop{gS}}
\newcommand{\cf}{\oop{cf}}

\newcommand{\ls}{\oop{LS}({\bf K})}

\newcommand{\llk}{\oop{L}({\bf K})}
\newcommand{\defeq}{\vcentcolon=}

\def\fork{\mathrel{\raise0.2ex\hbox{\ooalign{\hidewidth$\vert$\hidewidth\cr\raise-0.9ex\hbox{$\smile$}}}}}
\newcommand{\fk}[3]{#1 \underset{#2}{\fork} #3}
\newcommand{\fkc}[3]{#1 \underset{#2}{\bar{\fork}} #3}

\newcommand{\nr}[1]{\lVert #1 \rVert}
\newcommand{\card}[1]{\lvert #1 \rvert}
\newcommand{\al}{{\aleph_0}}

\makeatletter
\@ifpackageloaded{hyperref}%
  {\newcommand{\mylabel}[2]% #1=name, #2 = contents
    {\protected@write\@auxout{}{\string\newlabel{#1}{{#2}{\thepage}%
      {\@currentlabelname}{\@currentHref}{}}}}}%
  {\newcommand{\mylabel}[2]% #1=name, #2 = contents
    {\protected@write\@auxout{}{\string\newlabel{#1}{{#2}{\thepage}}}}}
\makeatother
\begin{document}

\pagenumbering{roman}
\setcounter{page}{0}
\newpage
\pagenumbering{arabic}
\setcounter{page}{1}
%\raggedtright
\parindent=.35in
\begin{center}
        \begin{center}%
         {\Large\bfseries\scshape Categoricity transfer for short AECs\\ with amalgamation over sets}\\\vspace{1em}{\scshape Samson Leung}\\      
        \end{center}%
\end{center}
{\let\thefootnote\relax\footnote{Date: \today\\
AMS 2010 Subject Classification: Primary 03C48. Secondary: 03C45, 03C55.
Key words and phrases. Abstract elementary classes; Categoricity; Good frames; Tameness; Shortness; Amalgamation.
}} \addtocounter{footnote}{-1}
\begin{abstract}
%In this chapter
Let ${\bf K}$ be an $\ls$-short abstract elementary class and assume more than the existence of a monster model (amalgamation over sets and arbitrarily large models). Suppose ${\bf K}$ is categorical in some $\mu>\ls$, then it is categorical in all $\mu'\geq\mu$. Our result removes the successor requirement of $\mu$ made by Grossberg-VanDieren \cite{GV06c}, at the cost of using shortness instead of tameness; and of using amalgamation over sets instead of over models. It also removes the primes requirement by Vasey \cite{s14} which assumes tameness and amalgamation over models. As a corollary, we obtain an alternative proof of the upward categoricity transfer for first-order theories \cite{morcat}\cite{sh31}. % as well as for compact AECs \cite{ss2} (the latter improved \cite{bshort}). 
In our construction, we simplify Vasey's results \cite{s6,s8} to build a weakly successful frame. This allows us to use Shelah-Vasey's argument \cite{ss2} to obtain primes for sufficiently saturated models. If we replace the categoricity assumption by $\ls$-superstability, ${\bf K}$ is already excellent for sufficiently saturated models. This sheds light on the investigation of the main gap theorem for uncountable first-order theories within ZFC.
\end{abstract}
\vspace{1em}
\begin{center}
{\bfseries TABLE OF CONTENTS}
\end{center}
\tableofcontents%chapter
\vspace{2em}
 
\section{Introduction}\label{apsec1}
For first-order theories, we have the following categoricity theorems:
\begin{theorem}
\begin{enumerate}
\item \mylabel{morsh}{Theorem \thetheorem}\emph{\cite{morcat}} Let $T$ be a countable first-order theory. If $T$ is categorical in some uncountable cardinal, then it is categorical in all uncountable cardinals.
\item \emph{\cite{sh31}} Let $T$ be a first-order theory. If $T$ is categorical in some uncountable cardinal, then it is categorical in all uncountable cardinals.
\end{enumerate}
\end{theorem}
%Shelah \cite[Conjecture N.4.2] {shh} conjectured that a generalization holds for abstract elementary classes (AECs):
%\begin{conj}[Eventual categoricity]
%Let $\lambda$ be an infinite cardinal. There exists a cardinal $\mu>\lambda$ such that for any %AEC ${\bf K}$ with $\ls=\lambda$, if ${\bf K}$ is categorical in some $\mu'\geq\mu$, then it is categorical in all $\mu'\geq\mu$.
%\end{conj}
In the late seventies after Shelah completed his book \cite{sh90}, he came up with a far reaching program: develop classification theory for non-elementary classes. Thus he titled his papers \cite{sh87a,sh87b,sh88} ``Classification theory for non-elementary classes''. In the summer of 1976, Shelah proposed as a test question for such a theory (which appeared in \cite[Conjecture 2]{sh87a}):
\begin{conj}[Categoricity conjecture for $L_{\omega_1,\omega}$]
Let $\psi$ be a sentence of $L_{\omega_1,\omega}$ in a countable language. If $\psi$ is categorical in some $\mu\geq\beth_{\omega_1}$, then $\psi$ is categorical in all $\mu\geq\beth_{\omega_1}$.
\end{conj}
In the second edition of his book \cite{sh90}, the conjecture was generalized to:
\begin{conj}[Categoricity conjecture for $L_{\lambda^+,\omega}$]
Let $\psi$ be a sentence of $L_{\lambda^+,\omega}$ in a language of size $\lambda$. If $\psi$ is categorical in some $\mu\geq\beth_{(2^\lambda)^+}$, then $\psi$ is categorical in all $\mu\geq\beth_{(2^\lambda)^+}$.
\end{conj}
In \cite[Section 6]{sh702}, Shelah stated that classification theory for abstract elementary classes (AECs) is the most important direction of model theory. He conjectured:
\begin{conj}[Categoricity conjecture for AECs]
Let ${\bf K}$ be an AEC and $\lambda=\ls$. The threshold for categoricity transfer is $\beth_{(2^\lambda)^+}$ (the Hanf number).
\end{conj}
%More precisely, we would like to find a class function $\lambda\mapsto\mu_\lambda$ such that if $\ls\leq\lambda$ and ${\bf K}$ is categorical in some $\mu\geq\mu_\lambda$, then it is categorical in all $\mu'\geq\mu$. 

The importance of these conjectures is the structural theory that needs to be developed. The main concept of the previously-developed structural theory for first-order theories is \emph{forking}: a canonical notion that generalizes \emph{combinatorial geometries} (also called \emph{matroids} when they are finitely generated). 

In about 3000 pages of publications towards these conjectures indeed such a theory evolved (see the \hyperlink{catlist}{table} at the end of this section for a partial list of results). We can divide the approaches into three types:
\begin{enumerate}
\renewcommand{\labelenumi}{\alph{enumi}.}
\item Assuming tameness and other model theoretic properties: Grossberg and VanDieren \cite{GV06c, GV06a} extracted the notion of tameness from \cite{sh394} and derived categoricity transfer from a successor cardinal for tame AECs with a monster model. Many subsequent results were obtained by Boney and Vasey but the successor assumption from \cite{GV06c} still could not be removed. Vasey \cite{s11} building upon Shelah's results, showed that categoricity transfer holds for AECs with amalgamation and primes (without starting from a successor cardinal) and managed to prove that the eventual categoricity conjecture is true for universal classes \cite{s8,s16}.
\item Assuming non-ZFC axioms and model theoretic properties: Shelah \cite{sh87a,sh87b} showed that under WGCH, if a countable theory in $L_{\omega_1,\omega}$ is excellent and has few models in $\aleph_n$ for $n<\omega$, then categoricity transfers up from an uncountable cardinal. \cite{shh} also developed heavy machineries such as good frames to derive categoricity transfers. However many of his results have technical assumptions which are not easy to verify. Later Shelah and Vasey \cite{ss2} generalized the notion of excellence to AECs and derived categoricity transfers assuming WGCH and restricting the spectrum in an interval of cardinals. A few variations were given in \cite{ss2,s31} where they replaced the spectrum requirements by other model theoretic properties.

Meanwhile, Makkai and Shelah \cite{ms90} proved that the eventual categoricity conjecture is true for an $L_{\kappa,\omega}$ theory starting at successor cardinals, where $\kappa$ is strongly compact. Boney \cite{bshort} showed that tameness holds for compact AECs (assuming the existence of strongly compact cardinals), thus by \cite{GV06c,GV06a} the eventual categoricity is true starting at successor cardinals. Eventually \cite{ss2} used the excellence argument to remove the successor assumption.
\item Using specific constructions: Cheung \cite{hanif} showed that given a free notion of amalgamation and the existence of prime models, the AEC behaves like strongly minimal theories, which allows one to manipulate the AEC algebraically.

Mazari-Armida \cite{mcat} combined decomposition results from algebra and categoricity transfer from \cite{s14} to characterize algebraically the property of being categorical in a tail. In particular, let $R$ be an associative ring with unity, he proved that the threshold of categoricity transfer is $(|R|+\al)^+$ for the class of locally pure-injective modules, flat modules and absolutely pure modules.

Esp\'{i}ndola \cite{esp1} used topos-theoretic argument to show that if the AEC has interpolation (which together with categoricity implies a stronger version of amalgamation), then the eventual categoricity conjecture holds (see also \cite{esp2}). However, there is no explicit bound to the threshold cardinal $\mu$. 
\end{enumerate}

In this paper we follow approach (a) above and focus on AECs that have a monster model, satisfy amalgamtion over sets and are type-short (a property that implies tameness, see \ref{shortdef}). In doing so we can remove the successor assumption in (2) in the \hyperlink{catlist}{table}. In our proof, we rely heavily on many recent papers and replace the use of WGCH in (8) by amalgamation over sets to obtain excellence. Then using \cite{ss2} that excellence implies primes, we can invoke the categoricity transfer in (3). A main application of this result is the removal of the successor requirement in the categorical transfer in \cite{ms90} (see also \cite[Question 6.14]{sh702} for the problem statement). %On the other hand, it is known that compact AECs have shortness, and with categoricity it has amalgamation over sets (see \ref{}). Hence we obtain an alternative proof of categoricity transfer of compact AECs in (6) (which improved (5)).

Our work was motivated by a simple question: using the common model theoretic assumptions and techniques, can we recover the upward\footnote{\emph{Downward} transfer is a much harder problem for AECs: the currently known transfer with common assumptions is down to the first Hanf number. \ref{lastexam} shows that the first categoricity cardinal can go up to the first Hanf number, but such example fails amalgamation and joint-embedding.} categoricity transfer in \ref{morsh}? \cite{les} and \cite{hk} have relevant results but they require $\ls=\al$ and several additional assumptions (say \emph{simplicity}: there is a strong example by Shelah that in the context of homogeneous model theory, simplicity is not a consequence of $\al$-stability \cite{hyle02}). Such results might not be easy to check and generalize to uncountable $\ls$. Meanwhile, Vasey \cite[Section 4]{s11} adopted a hybrid approach where he quoted syntactic results from \cite{sh3,HS00} to conclude that a homogeneous diagram has primes and a nonforking relation over sets, and then combined it with the categoricity transfer for AECs with amalgamation and primes. In comparison, our result is cleaner because we do not invoke primeness or stability results from \cite{sh3,HS00}. The assumptions of shortness and amalgamation over sets are immediate to check. 

When we show excellence, we only require shortness, amalgamation over sets, arbitrarily large models and superstability. This way of obtaining excellence does not use any non-ZFC axioms and might shed light on the main gap theorem for uncountable first-order theories: \cite{GLmain} used an axiomatic framework to obtain the abstract decomposition theorem, a key step to the main gap theorem. The results from \cite{ss2} provide us with a multidimensional independence relation, which satisfies some of the axioms in \cite{GLmain}. For future work, one may look at the axioms on regular types (see \cite[Axioms 8-10]{GLmain}).

We now list some of the known results on categoricity transfer for AECs. The numbering is for reference only and is not chronological. We strengthen some of the assumptions to a monster model ``$\mn$'' for readability (unless they assumed a local frame). Here a monster model means amalgamation, joint embedding and no maximal models. We write ``$\mn_{\text{set}}$'' if we also require amalgamation over sets. We strengthen instances of WGCH in an interval of cardinals to full WGCH. Throughout we let $\lambda=\ls$. Except for (5)(12), we assume that the categoricity cardinal $\mu<h(\lambda)$ (so we can omit the downward transfer to the first Hanf number $h(\lambda)$). Some of the results can be combined but we highlight the new parts. %The definition of $\mu({\bf K})$ is a locality cardinal and can be found in \cite[Definition 3.16]{hanif}. Compact AECs are defined in \ref{}.
The key results of categoricity transfers within ZFC are (2), (3) and (4). By assuming shortness and amalgamation over sets, we remove the successor assumption in (2) and (4), while removing the prime triples assumption in (3).

This paper was written while the author was working on a Ph.D. under the direction of Rami Grossberg at Carnegie Mellon University and we would like to thank Prof. Grossberg for his guidance and assistance in my research in general and in this work in particular.%We also thank John Baldwin, Hanif Cheung, Marcos Mazari-Armida and Wentao Yang for useful comments. 

\pagebreak
\begin{table}[t!]
%\begin{center}\underline{New and known results}\end{center}
\noindent\makebox[\textwidth]{
\begin{tabular}{|l|l|l|l|}
\hline
 \hypertarget{catlist}{}&Assumptions on ${\bf K}$&If $I(\mu,{\bf K})=1$ for some&Then $I(\mu',{\bf K})=1$ for all\\ \hline
 &$\lambda$-short, $\mn_{\text{set}}$&$\mu\geq\lambda^+$ & $\mu'\geq\mu$ (\ref{lastthm})\\ \hline
1. & Homogeneous diagram with $\mn_{\text{set}}$& $\mu\geq|T|^+$&$\mu'\geq\mu$ \cite[Theorem 4.22]{s11}\\ \hline
2. &$\lambda$-tame, $\mn$ & successor $\mu\geq\lambda^+$ & $\mu'\geq\mu$ \cite[Theorem 5.3]{GV06c}\\ \hline
3. &$\lambda$-tame, $\mn$, has primes&$\mu\geq\lambda^+$&$\mu'\geq\mu$ \cite[Theorem 10.9]{s14}\\\hline
4.& Has a type-full good $[\mu_1,\mu_2]$-&$\mu_1,\mu_2$ as on the left& $\mu'\in[\mu_1,\mu_2]$\\
& frame where $\mu_2$ is a successor $>\mu_1\geq\lambda$&  &\cite[Theorem 6.14]{s14}\\ \hline
5. &$\lambda<\kappa$ for some strongly compact $\kappa$&successor $\mu\geq\kappa^+$ & $\mu'\geq\mu$ \cite[Theorem 7.4]{bshort}\\ \hline
6. & Compact &  $\mu\geq\lambda^+$ & $\mu'\geq\mu$ \cite[Theorem 14.5]{ss2}\\ \hline
7.& Excellent & $\mu\geq\lambda^+$& $\mu'\geq\mu$ \cite[Theorem 14.2]{ss2}\\ \hline
8.& WGCH, has a $(<\omega)$-extendible& $\mu_2\geq\mu_1^+$&$\mu'\geq\mu_1^+$ \\ 
& categorical good $\mu_1$-frame&&\cite[Corollary 14.4]{ss2}\\\hline
9.&WGCH, $K_{\lambda^{++}}\neq\emptyset$ and for $n<\omega$,  &$\mu=\lambda,\lambda^+$ &$\mu'\geq\lambda$ \cite[Theorem 14.11]{ss2}\\ 
&$I(\lambda^{+n},{\bf K})<\mu_{\mathrm{unif}}(\lambda^{+n},2^{\lambda^{+(n-1)}})$ &&\\\hline
10.&WGCH, $\mn$ & $\mu_1,\mu_2\geq{\lambda}$& $\mu'\in[\mu_1,\mu_2]$ \cite[Lemma 9.5]{s31}\\ \hline
11.& WGCH, $\mn$ & $\mu>{\lambda}^{+\omega}$&$\mu'\geq\mu$ \cite[Lemma 9.6]{s31}\\ \hline
12. &Universal class & $\mu\geq\beth_{h(\lambda)}$ & $\mu'\geq\mu$ \cite[Theorem 7.3]{s16}\\ \hline
13.& $\mathrm{PC}_{\al}$, $\al$-tame, has primes, $2^{\aleph_0}<2^{\aleph_1}$&$\mu=\aleph_1$ &$\mu'\geq\mu$ \cite[Theorem 4.4]{sm}\\\hline
14.& WGCH, $\mathrm{PC}_{\al}$, $1\leq I(\aleph_1,{\bf K})<2^{\aleph_1}$, &$\mu\geq\aleph_1$ and $\mu=\aleph_0$&$\mu'\geq\al$\\
&and few models in $\aleph_n$&&\cite[Theorem 14.12]{ss2}\\\hline
15.&Atomic models of a countable first-order&$\mu\geq\aleph_1$&\cite{sh87a,sh87b}\\
&theory, WGCH, few models in $\aleph_n$&&\\\hline
16.& Universal $L_{\omega_1,\omega}$ sentence&Tail of $[\ls,\beth_\omega)$ &$\mu'\geq\beth_\omega$ \cite[Corollary 5.10]{s33}\\\hline
17. & Has prime and small models & $\mu\geq\mu({\bf K})+\lambda$ &$\mu'\geq\mu({\bf K})+\lambda+I(\lambda,{\bf K})^+$\\ 
 & and a free notion of amalgamation& &\cite[Theorem 5.7]{hanif}\\ \hline
18. & The class of  locally pure-injective modules/  & $\mu\geq(|R|+\al)^+$ & \cite{mcat}\\ 
& flat modules/absolutely pure modules& &\\ \hline
19. & Has interpolation & (no explicit bound) & \cite[Corollary 4.2]{esp1}\\ \hline
\end{tabular} }
\end{table}

\pagebreak

\section{Preliminaries}\label{apsec2}

In this section, we will define the main notions used in this paper (see \cite[Definition 2.2]{leung2} for the definition of AECs). Relevant results will be discussed in the subsequent sections. 

\begin{definition}\mylabel{setdef}{Definition \thetheorem}
Let ${\bf K}$ be an AEC and $\lambda\geq\ls$. The functions $f$ mentioned below will be ${\bf K}$-embeddings.
\begin{enumerate}
\item ${\bf K}$ has the \emph{$\lambda$-amalgamation property} ($\lambda$-$AP$) if for any $M_0,M_1,M_2\in K_\lambda$, $M_0\lk M_1$, $M_0\lk M_2$, there is $M_3\in K_\lambda$ and $f: M_1\xrightarrow[M_0]{}M_3$ such that $M_2\lk M_3$. ${\bf K}$ has the \emph{amalgamation property} ($AP$) when the above is true without the cardinal restriction.
\item ${\bf K}$ has the \emph{amalgamation property over set bases} (\emph{$AP$ over sets}) if for any $M_1,M_2\in K$, any $A\subseteq |M_1|\cap|M_2|$, there is $M_3\in K$ and $f: M_1\xrightarrow[A]{}M_3$ such that $M_2\lk M_3$. %such that the partial $\llk$-structure of $A$ in $M_1$ agrees with that in $M_2$ (namely, for any relation symbol $R$ and any function symbol $f$ in $\llk$, we have $R^{M_1}\restriction A=R^{M_2}\restriction A$ and $\oop{Gr}(f^{M_1})\restriction A=\oop{Gr}(f^{M_1})\restriction A$), then there is $M_3\in K$ and $f: M_1\xrightarrow[A]{}M_3$ such that $M_2\lk M_3$.
\item ${\bf K}$ has the \emph{$\lambda$-joint embedding property} ($\lambda$-$JEP$) if for any $M_1,M_2\in K$, there is $M_3\in K_\lambda$ and $f: M_1\rightarrow M_3$ such that $M_2\lk M_3$. ${\bf K}$ has the \emph{joint embedding property} ($JEP$) when the above is true without the cardinal restriction.
\item ${\bf K}$ has \emph{no maximal models} ($NMM$) if for any $M\in K$, there is $N\in K$ such that $M\lk N$ but $M\neq N$.
\item ${\bf K}$ has \emph{arbitrarily large models} ($AL$) if for any cardinal $\mu\geq\ls$, there is $M\in K_{\geq\mu}$.
\item ${\bf K}$ has a \emph{monster model} $\mn$ if it has $AP$, $JEP$ and $NMM$.
\item ${\bf K}$ has $\mn_{\text{set}}$ if it has $AP$ over sets (which implies $JEP$) and $NMM$. 
\end{enumerate}
\end{definition}
\begin{remark}
All the properties except for (2)(7) in the above definition hold in complete first-order theories because of the compactness theorem. (2)(7) hold if we also fix a monster model and they will only be used in Section \ref{apsec6}. See also the discussions around \cite[Definition 4.34, Lemma 18.8]{bal}.
\end{remark}
\begin{definition}
Let $\alpha\geq2$ be an ordinal. 
\begin{enumerate}
\item We denote \emph{Galois types} (\emph{orbital types}) of length $(<\alpha)$ as $\gs^{<\alpha}(\cdot)$ (see \cite[Definition 2.16]{s5}; we will not need the precise definition in this paper). The argument can be a set $A$ in some model $M\in K$. In general $\gs^{<\alpha}(A)\defeq\bigcup\{\gs^{<\alpha}(A;M):M\in K, |M|\supseteq A\}$ (under $AP$, the choice of $M$ does not matter).
\item ${\bf K}$ is \emph{$(<\alpha)$-stable in $\lambda$} if for any set $A$ in some model $M\in K$, $\card{A}\leq\lambda$, then $\card{\gs^{<\alpha}(A;M)}\leq\lambda$. We omit ``$(<\alpha)$'' if $\alpha=2$, while we omit ``in $\lambda$'' if there exists such a $\lambda\geq\ls$. Similarly ${\bf K}$ is \emph{$\alpha$-stable in $\lambda$} if for any such $A$ and $M$ above, we have $\card{\gs^\alpha(A)}\leq\lambda$. 
\end{enumerate}
\end{definition}

The notion of \emph{tameness} was introduced by Grossberg and VanDieren \cite{GV06c} as an extra assumption to an AEC. Later Boney \cite{bshort} introduced a dual property named \emph{shortness}. Tameness is a locality property on the domain of types while shortness is a locality property on the tuples that realize the types.
\begin{definition}\mylabel{shortdef}{Definition \thetheorem}
Let $\kappa$ be an infinite cardinal. 
\begin{enumerate}
\item Let $p=\gtp(a/A,N)$ where $a=\langle a_i:i<\alpha\rangle$ may be infinite, $I\subseteq\alpha$, $A_0\subseteq A$. We write $l(p)\defeq l(a)$, $p\restriction A_0\defeq \gtp(a/A_0,N)$, $a^I=\langle a_i:i\in I\rangle$ and $p^I\defeq\gtp(a^I/A,N)$.
\item $K$ is \emph{$(<\kappa)$-tame for $(<\alpha)$-types} if for any subset $A$ in some model of $K$, any $p\neq q\in\gs^{<\alpha}(A)$, there is $A_0\subseteq A$, $|A_0|<\kappa$ with $p\restriction A_0\neq q\restriction A_0$. We omit $(<\alpha)$ if $\alpha=2$.
\item $K$ is \emph{$(<\kappa)$-short} if for any $\alpha\geq2$, any subset $A$ in some model of $ K$, $p\neq q\in\gs^{<\alpha}(A)$, there is $I\subseteq\alpha$, $\card{I}<\kappa$ with $p^I\neq q^I$. 
\item \emph{$\kappa$-tame} means $(<\kappa^+)$-tame. Similarly for shortness.
\end{enumerate}
\end{definition}
\begin{remark}\mylabel{shortremark}{Remark \thetheorem}
By \cite[Corollary 3.18]{s5}, $(<\kappa)$-shortness implies $(<\kappa)$-tameness. First-order theories are trivially $(<\al)$-short, while a theorem due to Boney shows that universal classes are also $(<\al)$-short \cite[Theorem 3.7]{s8}. 
\end{remark}
The notion of a good frame was introduced in \cite[Chapter II]{shh}. The definition was extended for domains of sizes from an interval of cardinals (instead of a single cardinal) in \cite{s3} while for longer types in \cite{bv}. We follow the notation in \cite{bv} but specialize it in our context, where the types are always type-full (basic types coincide with nonalgebraic types) and we work inside a monster model. \cite{s6}, building on numerous papers, defined many more properties of a frame which cater for his coheir construction, which will not be considered here.

\begin{definition} Let ${\bf K}$ be an AEC with a monster model $\mn$, $\mu\geq\ls$ be a cardinal and $\alpha\geq2$ be an ordinal or $\infty$. A $(<\alpha,\geq\mu)$-good frame is a ternary relation $\fork$ such that:
\begin{enumerate}
\item If $(a,M_0,M_1)\in\ \fork$, then $a\in |M_1|^{<\alpha}$, $M_0\lk M_1$ and $M_0,M_1\in K_{\geq\mu}$. We write $\fk{a}{M_0}{M_1}$ and say $\gtp(a/M_1)$ \emph{does not fork over} $M_0$ (well-defined by invariance below).
\item (Invariance) If $f\in\oop{Aut}(\mn)$ and $\fk{a}{M_0}{M_1}$, then $\fk{f(a)}{f(M_0)}{f(M_1)}$.
\item (Monotonicity) If $\fk{a}{M_0}{M_1}$, $M_0\lk N_0\lk N_1\lk M_1$, $a'\subseteq a$ and $a'\in|N'|$, then $\fk{a'}{N_0}{N_1}$.
\item (Stability) For $M\in K_{\geq\mu}$, $|\gs(M)|\leq\nr{M}$.
\item (Existence) For $M\in K_{\geq\mu}$ and $a\in|M|^{<\alpha}$, $\fk{a}{M}{M}$.
\item (Extension) If $p\in\gs^{<\alpha}(M_1)$ does not fork over $M_0$, $M_1\lk M_2$ and $l(p)\leq\beta<\alpha$, then there is $q\in\gs^{\beta}(M_2)$ such that $q^{\beta}\restriction M=p$ and $q$ does not fork over $M_0$.
\item (Uniqueness) If $p,q\in\gs^{<\alpha}(M_1)$ do not fork over $M_0$ and $p\restriction M_0=q\restriction M_0$, then $p=q$.
\item (Transitivity) If $\fk{a}{M_0}{M_1}$ and $\fk{a}{M_1}{M_2}$, then $\fk{a}{M_0}{M_2}$.
\item (Local character) If $\delta$ is regular, $\langle M_i\in K_{\geq\mu}:i\leq\delta\rangle$ is increasing and continuous, $p\in\gs^{<\delta}(M_\delta)$, then there is $i<\delta$ such that $p$ does not fork over $M_i$.
\item (Continuity) If $\delta$ is a limit ordinal, both $\langle M_i\in K_{\geq\mu}:i\leq\delta\rangle$ and $\langle \alpha_i<\alpha:i\leq\delta\rangle$ both increasing and continuous, $p_i\in\gs^{\alpha_i}(M_i)$ increasing in $i<\delta$, then there is some $p\in\gs^{\alpha_\delta}(M_\delta)$ such that for all $i<\delta$, $p^{\alpha_i}\restriction M_i=p_i$. If each $p_i$ does not fork over $M_0$, then neither does $p$.
\item (Symmetry) If $\fk{a_2}{M_0}{M_1}$ and $a_1\in|M_1|^{<\alpha}$, then there is $M_2$ containing $a_2$ such that $\fk{a_1}{M_0}{M_2}$.
\end{enumerate}
We define $(<\alpha,\mu)$-frame similarly when the models must have size $\mu$. We omit ``$(<\alpha)$'' when $\alpha=2$. We call $\fork$ an independence relation if it only has invariance and monotonicity.
\end{definition}
\begin{remark}
There are weaker versions of a good frame which still have nice properties (for example \cite{JS13,mw,leung3}), which will not be discussed here because we will focus on the full strength of a good frame (and more under categoricity).
\end{remark}

\section{A $(<\infty,\geq(2^{\ls})^+)$-nonforking relation}\label{apsec3}
Assuming superstability and shortness, we will build a $(<\infty,\geq(2^{\ls})^+)$-nonforking relation with nice properties. This will allow us to use \cite[Section 11]{s6} to conclude that the underneath good frame is weakly successful. The result was sketched in \cite[Lemma A.14]{s8} but it drew technical results from \cite[Sections 1-10]{s6}. In this section, we will construct the nonforking relation and derive its properties directly. Readers can blackbox this section and skip to Section \ref{apsec4}.

\begin{definition}Let ${\bf K}$ be an AEC with a monster model, $\lambda\geq\ls$.
\begin{enumerate}
\item Let $M\lk N$, we say that $N$ is an \emph{universal extension of $M$} if for any $N'\in K_{\nr{M}}$ with $M\lk N'$, there is $f:N'\xrightarrow[M]{}N$. We say a chain $\langle M_i\in K_\lambda:i\leq\delta\rangle$ is \emph{universally increasing} if for each $i<\delta$, $M_{i+1}$ is a universal extension of $M_i$.
\item Let $N\in K$ and $p\in\gs(N)$, we say that $p$ $\lambda$-splits over $M$ if there exists $N_1,N_2\in K_\lambda$ such that $M\lk N_1,N_2\lk N$, $f:N_1\xrightarrow[M]{}N_2$ with $f(p)\restriction N_2\neq p\restriction N_2$.
\item ${\bf K}$ is \emph{superstable in $\lambda$} if ${\bf K}$ is stable in $\lambda$ and the following holds: for any limit ordinal $\delta<\lambda^+$, any universally increasing and continuous $\langle M_i\in K_\lambda:i\leq\delta\rangle$, $p\in\gs(M_\delta)$, there is $i<\delta$ such that $p$ does not $\lambda$-split over $M_i$.
\end{enumerate}
\end{definition}
\begin{remark}\mylabel{unisatrmk}{Remark \thetheorem}
In item (1), if $M\lk N$ and $N$ is $\nr{M}^+$-saturated, then $N$ is a universal extension over $M$. In addition, $N$ realizes all $(<\nr{M}^+)$-types over $M$. In item (3), by \cite[Proposition 10.10]{s6} or \cite[Corollary 6.11(1)]{leung3}, $\lambda$-tameness and $\lambda$-superstability imply $\lambda'$-superstability for $\lambda'\geq\lambda$.
\end{remark}
Under tameness and superstability, we can build a good frame in the successor cardinal. We remark that the original item (2) did not show whether the $\mu^+$-saturated models form an AEC (in particular whether they are closed under unions). It was only later in item (1) that the question was fully settled.
\begin{fact}\mylabel{s3fact}{Fact \thetheorem}
Let ${\bf K}$ be an AEC with a monster model and $\mu\geq\ls$. Suppose ${\bf K}$ is $\mu$-tame and superstable in $\mu$.
\begin{enumerate}
\item \cite[Corollary 6.10]{vv} For $\lambda>\mu$, ${\bf K}^{\lambda\text{-sat}}$ is an AEC with $\oop{LS}({\bf K}^{\lambda\text{-sat}})=\lambda$.
\item \cite[Theorem 7.1]{s3} The relation defined by: $p\in\gs(N)$ does not fork over $M\leq N$ if there is $M_0\in K_\mu$ such that $M$ is a universal extension over $M_0$ and $p$ does not $\mu$-split over $M_0$, induces a good $(\geq\mu^+)$-frame for ${\bf K}^{\mu^+\text{-sat}}$ (by (1) the $\mu^+$-saturated models form a sub-AEC of ${\bf K}$).
\end{enumerate}
\end{fact}
\begin{remark}
\begin{itemize}
\item Coheir in \cite{BG} is another candidate for a good frame, but one has to assume in addition the no weak order property and the extension property of coheir. To remove these assumptions, one has to raise the starting cardinal very high, so the threshold cardinal of categoricity transfer is way above $\mu^+$. See also item 2(a) after this remark.
\item One might wonder if it is possible to define the frame for $K_\mu$. \cite[Corollary 13.16]{snote} gave a weaker version where the underlying models are limit models while local character and continuity are for universally increasing chains (this argument was generalized to the strictly stable context in \cite{leung3}). Alternatively, \cite[Section 6]{s31} built a good $\mu$-frame by assuming WGCH and drawing heavily from  \cite{JS13} (WGCH is used to establish that the frame is weakly successful, see \ref{wsdef}). 
\end{itemize}
\end{remark}
Now we have a good $(\geq\mu)$-frame and would like to extend it to longer types. However, there are difficulties in terms of proving extension and local character. Besides the use of WGCH as in the above remark, we list three main approaches in literature:
\begin{enumerate}
\item Using independent sequences and tameness, \cite{bv} developed on \cite[Exercise III.9.4.1]{shh} to extend the frame to longer types. But such frame is not necessarily type-full, which is assumed in other results.
\item Extend the good $(\geq\mu)$-frame to a $(<\infty,\geq\mu)$-nonforking relation, which might not be a good frame itself. \cite[Section 11]{s6} gave sufficient conditions of the nonforking relation in order for the original frame to be weakly successful. Then one can quote \cite{JS13} to extend the original frame by $\oop{NF}$, which is a good frame. To build the nonforking relation, there are two ways:
\begin{enumerate}
\item \cite[Sections 1-10]{s6} built an axiomatic framework that allows one to use coheir to produce a good $(\geq\mu)$-frame (instead of using nonsplitting). To obtain the sufficient conditions above, he went on with a highly convoluted construction, which also uses canonicity to obtain properties from nonsplitting. Moreover, the threshold cardinal $\mu$ is very high (fixed points of the beth function) in order to use the no-order property. 
\item Using nonsplitting (\ref{s3fact}), \cite[Lemma A.14]{s8} sketched that it can be extended to a nonforking relation that satisfies the sufficient conditions. However, the details were sparse (about two paragraphs) and he invoked technical results from \cite[Sections 1-10]{s6}, which have numerous definitions and go back and forth between coheir and nonsplitting. 
\end{enumerate}
\end{enumerate}
We will adopt approach 2(b), but give an alternative proof that such nonforking relation satisfies the desired properties. In particular we do not need \cite{s6} in this section but refer to the simple construction in \ref{s3fact}(2). Our starting cardinal is $\mu^+$ for the same reason as the successor cardinal in \ref{s3fact}(2). Meanwhile \cite[Lemma A.14]{s8} starts at $\mu$, but we cannot verify the claims there. At the end it does not affect the categoricity transfer by virtue of \ref{lastfact}(2).
\begin{definition}\mylabel{forkccon}{Definition \thetheorem}
Let ${\bf K}$ be an AEC with a monster model, $\mu=2^{\ls}$ and assume ${\bf K}$ is $\ls$-short and superstable in $\ls$.
\begin{enumerate}
\item Since shortness implies tameness (\ref{shortremark}), we can define the nonforking relation as in \ref{s3fact}(2) but for $<({\ls}^+)$-types (instead of 1-types).  This is a $(<{\ls}^+,\geq\mu^+)$-nonforking relation $\fork$ over the $\mu^+$-saturated models.
\item Extend $\fork$ to a $(<\infty,\geq\mu^+)$-nonforking relation $\bar{\fork}$ by coheir : $\fkc{a}{M_0}{M_1}$ iff for any subsequence $a'\subseteq a$ of length $<{\ls}^+$, we have $\fk{a'}{M_0}{M_1}$.
\end{enumerate}
\end{definition}

The following collection of facts helps us establish local character properties. The second item below is from \cite[Theorem 3.5]{bon3.1}, which was usually cited as \cite[Theorem 3.1]{bon3.1} (the issue was clarified in \cite[Theorem 2.2]{leung1}). The statement of the third item can be found in \cite[Lemma A.12]{s8} and is essentially \cite[Fact 4.6]{GV06b}.
\begin{fact} \mylabel{longfact}{Fact \thetheorem} Let ${\bf K}$ be an AEC with a monster model, $\mu\geq\ls$ and $\alpha\geq1$.
\begin{enumerate}
\item \cite[Lemma 3.3]{sh394} If ${\bf K}$ is stable in $\mu$, $M\in K_{\geq\mu}$ and $p\in\gs(M)$, then there is $M_0\lk M$, $\nr{M}=\mu$ such that $p$ does not $\mu$-split over $M_0$. 
\item If ${\bf K}$ is stable in $\mu$ and $\mu=\mu^\alpha$, then it is $\alpha$-stable in $\mu$.
\item If $\kappa$ satisfies $\mu=\mu^{<\kappa}$, then item (1) is still true for $p\in\gs^{<\kappa}(M)$.
\end{enumerate}
\end{fact}
\begin{proof}
We sketch (3): by stability and (2), ${\bf K}$ is $(<\kappa)$-stable in $\mu$. The proof of (1) shows that if the conclusion of (1) fails, one can build a tree of types and models to contradict 1-stability in $\mu$, where ``1'' comes from $l(p)$. The same proof goes through for (3) because we now have $(<\kappa)$-stability in $\mu$.
\end{proof}

We now state the nice properties of $\bar{\fork}$ we constructed. Items (c) and (d) can be strengthened but they are sufficient for the next section. Notice that shortness is the key to obtain uniqueness in item (e) below.
\begin{proposition}\mylabel{forkcsum}{Proposition \thetheorem}
Let ${\bf K}$ be an AEC with a monster model, $\mu=2^{\ls}$ and assume ${\bf K}$ is $\ls$-short and $\ls$-superstable. The relation $\bar{\fork}$ defined in \ref{forkccon} satisfies the following:
\begin{enumerate}
\renewcommand{\labelenumi}{\alph{enumi}.}
\item $\bar{\fork}$ is a $(<\infty,\geq\mu^+)$-nonforking relation over the $\mu^+$-saturated models.
\item When restricted to 1-types, $\bar{\fork}$ is a good $(\geq\mu^+)$-frame.
\item For $n\geq2$, ${\bf K}^{\mu^{+n}\text{-sat}}$ is an AEC with $\oop{LS}({\bf K}^{\mu^{+n}\text{-sat}})=\mu^{+n}$. 
\item For $n\geq2$, $\bar{\fork}$ restricted to $(\leq\mu^{+n})$-types has local character for chains of length $\geq\mu^{+(n+1)}$. Namely, for any $a$ of length $(\leq\mu^{+n})$, any regular $\delta\geq\mu^{+(n+1)}$, any increasing and continuous chain $\langle M_i:i\leq\delta\rangle\subseteq K^{\mu^+\text{-sat}}$, there is $i<\delta$ such that $\fkc{a}{M_i}{M_\delta}$.
\item $\bar{\fork}$ has uniqueness.
\item $\bar{\fork}$ has the left $(\leq\mu^+)$-witness property: $\fkc{a}{M_0}{M_1}$ iff for any $a'\subseteq a$ of length $\leq\mu^+$, we have $\fkc{a'}{M_0}{M_1}$.
\item $\bar{\fork}$ has the right $(\leq\mu^+)$-model witness property:  $\fkc{a}{M_0}{M}$ iff for any $M_1\in K^{\mu^+\text{-sat}}$ with $M_0\lk M_1\lk M$, $\nr{M_1}\leq\mu^+$, we have $\fkc{a}{M_0}{M_1}$.
\end{enumerate}
\end{proposition}
\begin{proof}
Items (a) and (b) follow from the construction of $\bar{\fork}$ which extends the original frame. Item (c) is by \ref{s3fact}(1). 

For item (d), we first assume that $a$ has length $<{\ls}^+$. Since $\mu=\mu^{<{\ls}^+}$, by \ref{longfact}(3) there is $M^*\lk M_\delta$, $\nr{M^*}=\mu$ such that $\gtp(a/M_\delta)$ does not $\mu$-split over $M^*$. Since $\delta\geq\mu^{+n}>\mu$, there is $i<\delta$ such that $M^*\lk M_i$. Since $M_i$ is $\nr{M^*}^+$-saturated, by \ref{unisatrmk} $M_i$ is universal over $M^*$. By definition, $\fk{a}{M_i}{M_\delta}$ as desired. Now for general $a$ of length $(\leq\mu^{+n})$, there are at most $(\mu^{+n})^{\ls}$, which is $\mu^{+n}$ many subsequences of length $<{\ls}^+$, therefore we can take the maximum $i$ from the previous case, which is still less than $\delta$ by a cofinality argument.

For item (e), let $M\lk N\in K^{\mu^+\text{-sat}}$, $p,q\in \gs^{<\infty}(N)$ both do not fork over $M$ and $p\restriction M=q\restriction M$. By shortness we may assume that $p,q\in\gs^{<{\ls}^+}(N)$. Then the uniqueness proof for the case of 1-types in \ref{s3fact}(2) goes through, because it uses universal extensions only and our types $p,q$ have length $<{\ls}^+$ less than the sizes of the models.  

Item (f) is true by coheir in the construction, in particular we have $(\leq\ls)$-witness property which is stronger. We show the backward direction of item (g): by coheir and monotonicity, it suffices to consider the case $l(a)<{\ls}^+$. By \ref{longfact}(3), there is $M^*\lk M$, $\nr{M^*}=\mu$ such that $\gtp(a/M)$ does not $\mu$-split over $M^*$. Pick $N_0\in K^{\mu^+\text{-sat}}$ such that $M_0\lk N_0\lk M$ and $N_0$ is a universal extension over $M^*$. By definition, $\gtp(a/M)$ does not fork over $N_0$. Since $\nr{N_0}=\mu^+$, by assumption $\gtp(a/N_0)$ does not fork over $M_0$. Now we can quote the transitivity proof for the case of 1-types in \ref{s3fact}(2), which generalizes to $<{\ls}^+$-types for the same reason as in the previous paragraph. Thus we have $\gtp(a/M)$ does not fork over $M_0$ as desired.
\end{proof}
\section{A weakly successful frame}\label{apsec4}
By \ref{forkcsum}, we will show that the nonforking relation in \ref{forkccon} satisfies \cite[Hypothesis 11.1]{s6}. This allows us to quote results from \cite[Sections 11, 12]{s6} and conclude that the underlying good $(\geq(2^{\ls})^+)$-frame is weakly successful, can be extended by $\oop{NF}$, is $\omega$-successful and has full model continuity (in the third successor cardinal). This will allow us to do categoricity transfer in Section \ref{apsec6}. On the other hand, we compare our extended frame with the results in \cite[Section 15]{s6}, which was constructed from coheir (instead of nonsplitting). 
\begin{proposition}\mylabel{checkprop}{Proposition \thetheorem}
Let ${\bf K}$ be an AEC with a monster model, $\mu=2^{\ls}$ and assume ${\bf K}$ is $\ls$-short and $\ls$-superstable. The relation $\bar{\fork}$ defined in \ref{forkccon} satisfies \cite[Hypothesis 11.1]{s6}.
\end{proposition}
\begin{proof}
The hypothesis is a list of requirements on the nonforking relation $\bar{\fork}$. By substituting ``$\lambda$'' and ``$\mu$'' there by $\mu^{++}$ and $\mu^+$ respectively. We check the items in the same numbering as in the hypothesis.
\begin{enumerate}
\item This is exactly \ref{forkcsum}(a). There they use the term ``independence relation'' to allow the right hand side of $\bar{\fork}$ to be sets (instead of models), which is just a generalization and does not affect the rest of the proof.
\item This is \ref{forkcsum}(b).
\item By the substitution above, clearly $\mu^{++}>\mu^+$.
\item This is \ref{forkcsum}(c)(d).
\item Base monotonicity is built in our definition of nonforking relation. Uniqueness is by \ref{forkcsum}(e).
\item This is \ref{forkcsum}(f)(g).
\end{enumerate}
\end{proof}
Under \cite[Hypothesis 11.1]{s6}, Vasey imitated the proofs in \cite{ms90} and showed that the underlying good $(\geq\mu^{++})$-frame has domination triples (see \ref{domdef}). Then he connected domination triples with uniqueness triples, which allowed him to conclude that the frame is weakly successful. In the following we state the relevant definitions and results.

The term ``domination triples'' came from the later \cite[Definition A.17]{s8} and \cite[Definition 2.9]{s12} even though \cite[Definition 11.5]{s6} had already investigated the idea of domination.
\begin{definition}\mylabel{domdef}{Definition \thetheorem}
Let $\lambda>\ls$ and $\fork$ be a $(<\infty,\geq\lambda)$-nonforking relation over the $\lambda$-saturated models. 
\begin{enumerate}
\item A triple $(a,M,N)$ is a \emph{domination triple} if $M\lk N$ both $\lambda$-saturated, $a\in|N|\backslash|M|$ and for any $\lambda$-saturated $N'$, $\fk{a}{M}{N'}$ implies $\fk{N}{M}{N'}$.
\item $\fork$ has the \emph{$\lambda$-existence property for domination triples} if for any $M$ saturated in $K_\lambda$, any nonalgebraic $p\in\gs(M)$, there exists a domination triple $(a,M,N)$ such that $p=\gtp(a/M;N)$.
\end{enumerate}
\end{definition}

The following fact \cite[Lemma 11.12]{s6} shows the existence property for domination triples. It will be applied to \ref{firsttrans} to show that the sufficiently saturated models have primes.

\begin{fact}\mylabel{domexist}{Fact \thetheorem}
In \ref{checkprop}, for $\lambda>\mu^+$, $\bar{\fork}$ has the $\lambda$-existence property for domination triples. 
\end{fact}

Now we look at uniqueness triples and weak successfulness.

\begin{definition}\cite[Definition 11.4]{s6}.\mylabel{wsdef}{Definition \thetheorem}
Let $\lambda>\ls$ and $\fork$ be a good $\lambda$-frame over the saturated models in $K_\lambda$. Let $M_0\lk M_1$ and $M_0\lk M_2$ all $\lambda$-saturated.
\begin{enumerate}
\item An \emph{amalgam} of $M_1$ and $M_2$ over $M_0$ is a triple $(f_1,f_2,N)$ such that $N$ is $\lambda$-saturated, $f_i:M_i\xrightarrow[M_0]{}N$ for $i=1,2$.
\item Two amalgams $(f_1^a,f_2^a,N^a)$, $(f_1^b,f_2^b,N^b)$ of $M_1$ and $M_2$ over $M_0$ are \emph{equivalent} if there are $N\in K_\lambda^{\lambda\text{-sat}}$, $f^a:N^a\rightarrow N$ and $f^a:N^a\rightarrow N$ such that the following diagram commutes:
\begin{center}
\begin{tikzcd}
                                                         & N^b \arrow[r, "f^b", dotted]                  & N                            \\
M_1 \arrow[rr, "f_1^a\qquad\qquad"'] \arrow[ru, "f_1^b"] &                                               & N^a \arrow[u, "f^a", dotted] \\
M_0 \arrow[r] \arrow[u]                                  & M_2 \arrow[uu, "f_2^b"'] \arrow[ru, "f_2^a"'] &                             
\end{tikzcd}
\end{center}
\item A triple $(a,M,N)$ is a \emph{uniqueness triple} if $M,N$ are saturated models in $K_\lambda$, $a\in|N|\backslash|M|$ and for any $M_1$ saturated in $K_\lambda$, there exists an amalgam $(f_1,f_2,N_1)$ of $N$ and $M_1$ over $M$ such that $\gtp(f_1(a)/f_2[M_1];N_1)$ does not fork over $M$ and the amalgam is unique up to equivalence (see item (2)).
\item $\fork$ is \emph{weakly successful} if it has the \emph{existence property} for uniqueness triples: for any $M$ saturated in $K_\lambda$, any nonalgebraic $p\in\gs(M)$, we can find a uniqueness triple $(a,M,N)$ such that $p=\gtp(a/M;N)$.
\end{enumerate}
\end{definition}
The following fact translates \cite[Theorem 11.13]{s6} into our context.
\begin{fact}\mylabel{wsfact}{Fact \thetheorem}
Under \cite[Hypothesis 11.1]{s6}, the relation $\bar{\fork}$ defined in \ref{forkccon} (when restricted to 1-types and $\mu^{++}$-saturated models) induces a weakly successful good $\mu^{++}$-frame over the $\mu^{++}$-saturated models.
\end{fact}
\begin{corollary}\mylabel{wscor}{Corollary \thetheorem}
Let ${\bf K}$ be an AEC with a monster model and $\mu=2^{\ls}$. Suppose ${\bf K}$ is $\ls$-short and superstable in $\ls$. Then the good $(\geq\mu^+)$-frame defined in \ref{s3fact}(2) induces a weakly successful good $\mu^{++}$-frame over the $\mu^{++}$-saturated models.
\end{corollary}
\begin{proof}
Since ${\bf K}$ is $\ls$-short and superstable in $\ls$, it is also $\mu$-short and superstable in $\mu$ and we can use \ref{s3fact}(2) to build a good $(\geq\mu^+)$-frame $\fork$. By \ref{forkccon}, \ref{forkcsum} and \ref{checkprop}, we can extend $\fork$ to a nonforking relation $\bar{\fork}$ that satisfies \cite[Hypothesis 11.1]{s6}. By \ref{wsfact}, $\bar{\fork}$ induces a weakly successful good $\mu^{++}$-frame over the $\mu^{++}$-saturated models. But this frame is just $\fork$ restricted to $\mu^{++}$-saturated models. 
\end{proof}

One more ingredient for categoricity transfer is the property of full model continuity. Vasey drew results from \cite{shh,JS13,Jar16} and showed that the weakly successful frame we obtained is $\omega$-successful. And if we move up by three successors (so we consider $\mu^{+5}$-saturated models), then it can be extended to a good frame with full model continuity. 

\begin{definition}
Let ${\bf K}$ be an AEC with a monster model, $\lambda\geq\ls$ and $\fork$ be a $(<\infty,\geq\lambda)$-nonforking relation on $K_{\geq\lambda}$. $\fork$ has \emph{full model continuity} if the following holds: for any limit ordinal $\delta$, any $\langle M_i^k:i\leq\delta\rangle$ increasing and continuous in $K_{\geq\lambda}$ where $k=0,1,2$, if $\fk{M_i^1}{M_i^0}{M_i^2}$ for each $i<\delta$, then $\fk{M_\delta^1}{M_\delta^0}{M_\delta^2}$.
\end{definition}

We sum up the previous paragraph in the following fact. The original results were from \cite[Sections 11, 12]{s6} but applied them to our context (in the same spirit as \ref{wscor}). In particular item (1) is from \cite[Theorem 11.21]{s6}; item (2) is from \cite[Theorem 12.16]{s6}. We will not define \emph{$\omega$-successfulness} because under amalgamation and tameness, it coincides with weak successfulness \cite[Facts 11.15, 11.19]{s6}. Also, \emph{good+} will be automatically satisfied by the new frame \cite[Fact 11.17]{s6} so we skip its definition.

\begin{fact}\mylabel{sfact}{Fact \thetheorem}
Let ${\bf K}$ be an AEC with a monster model and $\mu=2^{\ls}$. Suppose ${\bf K}$ is $\ls$-short and superstable in $\ls$.\begin{enumerate}
\item The weakly successful good $\mu^{++}$-frame from \ref{wscor} is also $\omega$-successful.
\item Let $\lambda=(\mu^{++})^{+3}=\mu^{+5}$. The frame can be extended by $\oop{NF}$ (defined for quadruples of models) and then closed to a good $(\leq\lambda,\geq\lambda)$-frame over the $\lambda$-saturated models. Moreover, the new frame is good+ and has full model continuity.
\end{enumerate}
\end{fact}

The rest of this section discusses what happens if we combine our results with \cite[Sections 13-15]{s6}. Readers only interested in categoricity transfer can skip to \ref{primefact} which will be used in Section \ref{apsec6}.

After obtaining a good $(\leq\lambda,\geq\lambda)$-frame with full model continuity, Vasey \cite[Sections 13,14]{s6} went on extending the right hand side of $\fork$ to arbitrary sets, and then the left hand side to arbitrary lengths. Such results still apply to our construction because we have shortness and amalgamation in our background assumptions (see also \cite[Hypotheses 13.1, 14.1]{s6}). We first state what Vasey had obtain in \cite[Theorem 15.6]{s6}.

\begin{fact}\mylabel{dtlfact}{Fact \thetheorem}
Let ${\bf K}$ be a $(<\kappa)$-short AEC with a monster model. Suppose there are $\lambda,\theta$ such that 
\begin{enumerate}
\item $\ls<\kappa=\beth_\kappa<\lambda=\beth_\lambda\leq\theta$;
\item $\cf(\lambda)\geq\kappa$;
\item ${\bf K}$ is categorical in $\theta$;
\end{enumerate}
then there is a $(<\infty,\geq\lambda^{+4})$-good frame over the $\lambda^{+4}$-saturated models except that extension holds over saturated models only. Moreover it has full model continuity.
\end{fact}
We state one more fact from \cite{sh394} about categoricity. A complete proof can be found in \cite{BGVV}.
\begin{fact}\mylabel{catfat}{Fact \thetheorem}
Let ${\bf K}$ be an AEC with a monster model. Suppose ${\bf K}$ is categorical in some $\lambda>\ls$, then ${\bf K}$ is superstable in $\ls$.
\end{fact}
To compare \ref{dtlfact} with our results, we replace our assumptions of $\ls$-shortness by $\kappa$-shortness, and superstability in $\ls$ by superstability in $\kappa$. 
\begin{corollary}\mylabel{finalcor}{Corollary \thetheorem}
Let ${\bf K}$ be a $\kappa$-short AEC with a monster model where $\kappa\geq\ls$. Suppose ${\bf K}$ is categorical in some $\theta>\kappa$ (superstability in $\kappa$ is sufficient), then there is a $(<\infty,\geq(2^{\kappa})^{+5})$-good frame over the $(2^{\kappa})^{+5}$-saturated models models except that extension holds over saturated models only. Moreover it has full model continuity.
\end{corollary}
\begin{proof}[Proof sketch]
By categoricity and \ref{catfat}, ${\bf K}$ is superstable in $\kappa$. By \ref{sfact} (replacing $\ls$ there by $\kappa$), there is a $(<(2^{\kappa})^{+5},\geq(2^{\kappa})^{+5})$-good frame over the $(2^{\kappa})^{+5}$-saturated models. Extend the frame to arbitrarily long types as in \cite[Sections 13,14]{s6}.
\end{proof}
As we can see, using nonsplitting to build a good frame has a much lower threshold than using coheir in obtaining \ref{dtlfact}. The fixed points of beth function are to guarantee no order property (see \cite[Fact 2.21]{s6}), which currently lacks a good upper bound (under amalgamation and stability). \cite[Corollary A.16]{s8} claimed a result similar to our corollary and we highlight the differences here: 
\begin{enumerate}
\item The threshold he obtained is $({\ls}^{<\kappa})^{+5}$ while ours is $(2^{\kappa})^{+5}$.
\item He used $(<\kappa)$-shortness directly but we weakened it to $\kappa$-shortness. We did so both for convenience and to readily apply \ref{sfact}. 
\item In verifying \cite[Hypothesis 11.1]{s6}, he drew heavy machineries from \cite[Sections 1-10]{s6} but we proved them directly in \ref{checkprop}.
\end{enumerate}

\section{Primes for saturated models}\label{apsec5}
We will combine the results from the previous section and \ref{dombuild} below to conclude that ${\bf K}$ has primes for saturated models. However, it is not clear whether this implies primes for models in general, so we cannot invoke categoricity transfer of AECs with primes and amalgamation. Readers only interested in categoricity transfer can skip to \ref{primefact} which will be used in the next section.
\begin{definition}\cite[Definition 2.13]{s12}
Let ${\bf K}$ be an AEC. 
\begin{enumerate}
\item A triple $(a,M,N)$ is a \emph{prime triple} if $M\lk N$, $a\in|N|\backslash|M|$, and the following holds: for any $N'\in K$ with $a'\in|N'|$ and $\gtp(a/M;N)=\gtp(a'/M;N')$ then there exists $f:N\xrightarrow[M]{}N'$ such that $f(a)=a'$.
\item ${\bf K}$ \emph{has primes} if for each $M\in K$ and each nonalgebraic $p\in\gs(M)$, there exists a prime triple $(a,M,N)$ such that $p=\gtp(a/M;N)$.
\end{enumerate}
\end{definition}
The original statement of the following fact is about $K^*$ only but we strengthen the monster model assumption to $K$. Vasey allowed the right hand side of $\fork$ to be sets (and had extra axioms) but we stick to models (see also the proof of \ref{checkprop}(1)).
\begin{fact}\cite[Theorem 3.6]{s12}\mylabel{dombuild}{Fact \thetheorem}
Let ${\bf K}$ be an AEC with a monster model. Suppose there is $\lambda_0\geq\ls$ and ${\bf K}^*$ such that:
\begin{enumerate}
\item ${\bf K}^*\subseteq {\bf K}$ is a sub-AEC of ${\bf K}$;
\item ${\bf K}^*$ is categorical in $\lambda_0$;
\item There is a good $(<\infty,\geq\lambda_0)$-frame with full model continuity over ${\bf K}^*$;
\item ${\bf K}^*_{\lambda_0}$ has the $\lambda_0$-existence property for domination triples (see \ref{domdef});
\end{enumerate}
Then for any $\lambda>\lambda_0$, the saturated models of ${\bf K}^*_{\lambda}$ has primes.
\end{fact}
\begin{corollary}\mylabel{firsttrans}{Corollary \thetheorem}
Let ${\bf K}$ be an AEC with a monster model and $\lambda_0=(2^{\ls})^{+5}$. Suppose ${\bf K}$ is $\ls$-short and superstable in $\ls$, then for $\lambda>\lambda_0$, ${\bf K}_\lambda^{\lambda\text{-sat}}$ has primes.
\end{corollary}
\begin{proof}
Let ${\bf K}^*=K^{\lambda_0\text{-sat}}$. ${\bf K^*}$ is a sub-AEC of ${\bf K}$ by \ref{s3fact}(1) and is categorical in $\lambda_0$ by a back-and-forth argument. Substituting $\kappa=\ls$ in \ref{finalcor}, there is a good $(<\infty,\geq\lambda_0)$-frame with full model continuity over ${\bf K}^*$. We would like to invoke \ref{domexist} (substituting $\lambda$ there by $\lambda_0$) and say that the good frame has $\lambda_0$-existence property for domination triples. While the good frame might not agree with the nonforking relation in \ref{domexist} for longer types, they both extend the good $(<2,\geq\lambda_0)$-frame from \ref{s3fact}(2). Since domination triples are about 1-types only, we can conclude that the nonforking relation from \ref{domexist} and hence the good frame from \ref{finalcor} has the $\lambda_0$-existence property for domination triples. By \ref{dombuild}, for $\lambda>\lambda_0$, $({\bf K}^*)^{\lambda\text{-sat}}_{\lambda}={\bf K}_\lambda^{\lambda\text{-sat}}$ has primes.
\end{proof}
\begin{remark}
The above proof went back to the notion of domination triples (instead of uniqueness triples) to quote \ref{domexist} because it was used in the assumptions of \cite{s12}. We suspect that one can derive a version of \ref{dombuild}(4) with uniqueness triples, which can simplify the proof because we have the existence property of the latter (see \ref{wsfact}). In the original construction, \cite[Section 11]{s6} built domination triples and showed that they are also uniqueness triples. \cite[Remark 11.8]{s6} claimed that if the nonforking relation has extension (to longer types), then uniqueness triples are domination triples. \cite[Fact A.18]{s8} cited \cite[Lemma 11.7]{s6} without proof that it is true in general (without assuming extension). We cannot verify those claims so we follow the longer route to obtain the existence property for domination triples.
\end{remark}
It would be ideal if \ref{firsttrans} concluded that ${\bf K}_\lambda$, instead of ${\bf K}_\lambda^{\lambda\text{-sat}}$, has primes, because we have the following fact:
\begin{fact}\mylabel{primefact}{Fact \thetheorem}\cite[Corollary 10.9]{s14}
Let ${\bf K}^*$ be an $\oop{LS}({\bf K}^*)$-tame AEC with primes and arbitrarily large models. If ${\bf K}^*$ is categorical in some $\lambda>\oop{LS}({\bf K}^*)$, then it categorical in all $\lambda'\geq\min(\lambda,h(\oop{LS}({\bf K}^*)))$.
\end{fact}
The main component of the proof came from \cite{s8} (or see \cite{snote} for a written-up  version). The idea is that ${\bf K}$ to show categoricity $\lambda'>\lambda$, one can pick a bigger categorical cardinal $\lambda''$ (guaranteed by \cite[Theorem 9.8]{s14}). Suppose $K^*_{\lambda'}$ is not categorical, then one can use primes to transfer non-saturation from $\lambda'$ to $\lambda''$. Since we cannot assume $K^*_{\lambda'}$ is categorical in the first place, we need primes for $K^*_{\lambda'}$ rather than the saturated models of $K^*_{\lambda'}$.

\begin{ques}
Using the assumptions in \ref{firsttrans} (or more), is it possible to obtain primeness for sufficiently saturated models? A positive answer will simplify the rest of the proof and remove the assumption of amalgamation over sets to obtain categoricity transfer.
\end{ques}
\section{$AP$ over sets and multidimensional diagrams}\label{apsec6}
In this section, we will add the extra assumption of amalgamation over sets (\ref{setdef}) to obtain excellence over sufficiently saturated models. This allows us to use \cite{ss2} and show that those models have primes. Then we can invoke \ref{primefact} to do categoricity transfer.

In \cite[Section 7]{ss2}, given a categorical good $\lambda$-frame (for example a good frame over the $\lambda$-saturated models), they defined when a frame \emph{reflects down}, is \emph{extendible}, \emph{very good} etc. We do not need the precise definitions but only the following fact:
\begin{fact}\mylabel{extfact}{Fact \thetheorem}
Let ${\bf K}$ be a $\ls$-short AEC with a monster model. Suppose ${\bf K}$ is superstable in $\ls$ and let $\lambda=(2^{\kappa})^{+5}$, then there is a $(<\omega)$-extendible categorical good $(\geq\lambda)$-frame over the $\lambda$-saturated models.
\end{fact}
\begin{proof}[Proof sketch]
Readers familiar with \cite{ss2} and Vasey's papers can consult \cite[Fact 7.21]{ss2}, which applied the same idea on compact AECs. Notice that ``${\ls}^{+6}$'' there should be $\kappa^{+6}$.\\ Alternatively, we use the frame from \ref{finalcor} and verify directly the extra conditions (see \cite[Section 7]{ss2} for relevant definitions):
\begin{enumerate}
\item There is a two-dimensional nonforking relation that extends our frame: this is witnessed by $\oop{NF}$ in \ref{sfact}(2).
\item The two-dimensional nonforking relation is \emph{good}: namely the frame it extends is a good frame; the nonforking relation has long transitivity and local character. Our $\oop{NF}$ satisfies these by \cite[Facts 12.2, 12.10]{s6}.
\item The two-dimensional nonforking relation \emph{reflects down}: by \cite[Remark 7.8]{ss2} it suffices to check that it is good and extends to $\lambda^+$. This is true again by \ref{sfact}(2).
\item The two-dimensional nonforking relation has full model continuity (which makes the relation \emph{very good}). This is true by \ref{sfact}(2).
\item The frame is $(<\omega)$-extendible: by \cite[Fact 7.20]{ss2}, it suffices to show that it is  $\omega$-successful and good+, which is true by \ref{sfact}(1)(2).
\end{enumerate}
\end{proof}

Given a $(<\omega$)-extendible good frame, \cite[Sections 8-11]{ss2} went on to build multidimensional independence relations from the two-dimensional nonforking relation (which extends the good frame). Basically a multidimensional independence relation takes in models indexed by a general partial order instead of $\mathcal{P}(2)$ as in a two-dimensional nonforking relation (see \cite[Definition 8.11]{ss2} for a precise definition). We state some relevant definitions:
\begin{definition}Let ${\bf K}$ be an abstract class and $(I,\leq)$ be a partial order.
\begin{enumerate}
\item \cite[Definition 8.1]{ss2}
 An \emph{$(I,{\bf K})$-system} is a sequence ${\bf m}=\langle M_u:u\in I\rangle$ such that $u\leq v\Rightarrow M_u\lk M_v$. We omit ${\bf K}$ if the context is clear. Usually $I=\mathcal{P}(n)$ or $I=\mathcal{P}(n)\backslash\{n\}$ for some $n<\omega$.
\item \cite[Definition 8.8]{ss2}
The language of $(I,{\bf K})$-systems is $\tau^I\defeq L({\bf K})\cup\{P_i:i\in I\}$ where each $P_i$ is a unary predicate. The \emph{abstract class of $(I,{\bf K})$-systems} is ${\bf K}^I=(K^I,\leq_{K^I})$ where for each ${\bf m}\in K^I$, $\langle (P_i)^M:i\in I\rangle$ forms a disjoint system and the models in $K^I$ are ordered by disjoint extensions (see \cite[Definition 8.6]{ss2}).
\end{enumerate}
\end{definition}
\begin{remark}
For our purpose, we only need to know that if ${\bf m}\in K^{\mathcal{P}(n)}$, then $\langle(P_i)^{\bf m}:i\in \mathcal{P}(n)\rangle$ is an $\mathcal{P}(n)$-system whose models are at least ordered by $\lk$.
\end{remark}
We now define a generalized version of amalgamation as well as higher-dimensional uniqueness properties. These were key to establish excellence and to build primes.
\begin{definition}
\begin{enumerate}
\item
\cite[Definition 5.6]{ss2} Let ${\bf K}$ be an abstract class in $\tau$ and let $\phi$ be a first-order quantifier-free formula in $\tau$.
\begin{enumerate}
\item $M,N\in K$ are \emph{$\phi$-equal} if $\phi(M)=\phi(N)$ and the induced partial $\tau$-structures by $\phi$ on $M,N$ are equal: for each relation and function symbol $R\in\tau$, $R^M\restriction\phi(M)=R^N\restriction \phi(N)$.
\item A \emph{$\phi$-span} is a triple $(M_0,M_1,M_2)$ such that $M_0\lk M_1$, $M_0\lk M_2$; and $M_1,M_2$ are $\phi$-equal.
\item A \emph{$\phi$-amalgam} of a $\phi$-span $(M_0,M_1,M_2)$ is a triple $(N,f_1,f_2)$ such that $N\in K$, $f_i:M_i\xrightarrow[M_0]{}N$ for $i=1,2$ and $f_1\restriction\phi(M_1)=f_2\restriction\phi(M_2)$.
\item $M\in K$ is a \emph{$\phi$-amalgamation base} if every $\phi$-span of the form $(M,M_1,M_2)$ has a $\phi$-amalgam.
\end{enumerate}
\item \cite[Definition 10.14]{ss2} For $n<\omega$, let $\phi_n$ be the formula in the language of $(n,{\bf K})$-systems such that for any ${\bf m}=\langle M_u:u\in \mathcal{P}(n)\rangle$, $a\in|{\bf m}|$, we have ${\bf m}\vDash\phi_n[a]$ iff $a\in\bigcup_{u\in\mathcal{P}(n)\backslash\{n\}}M_u$.
\item \cite[Definition 10.2]{ss2} 
\begin{enumerate}
\item For $n<\omega$, let $\mathcal{I}_n$ be the class of all partial orders isomorphic to an initial segment of $\mathcal{P}(n)$ and let $\mathcal{I}_{<\omega}=\bigcup_{n<\omega}\mathcal{I}_n$.
\item Let $\mathfrak{i}$ be a multidimensional independence relation and $P$ be either existence, extension or uniqueness (see \cite[Definitions 8.11, 8.16]{ss2}; we do not need the precise descriptions here). Let $I\subseteq \mathcal{I}_{<\omega}$ be a partial order and $\lambda\geq\ls$.
\begin{enumerate}
\item $\mathfrak{i}$ has $n$-$P$ if $I$ is defined on $\mathcal{P}(n)$-systems and $\mathfrak{i}\restriction \mathcal{I}_n$ has $P$. 
\item $\mathfrak{i}$ has $(\lambda,n)$-$P$ if $\mathfrak{i}\restriction K_\lambda$ has $n$-$P$.
\end{enumerate}
\end{enumerate}
\end{enumerate}
\end{definition}
We will adapt the proof of item (2) below to transfer uniqueness to higher dimensions. They used WGCH and we will replace it by amalgamation over sets. The construction of ${\bf K}_{\mathfrak{i},\mathfrak{i}^*,\mathcal{P}(n)}^{\text{proper},*}$ is very complicated and spans several sections. We only need to know that it is a sub-abstract class of ${\bf K}^{\mathcal{P}(n)}$.
\begin{fact}\mylabel{stepfact}{Fact \thetheorem}
Let $n<\omega$, $\mathfrak{i}$ be a very good (see \cite[Definition 11.2]{ss2}) multidimensional independence relation defined on $\mathcal{P}(n+1)$-systems, $\mathfrak{i}^*$ be its restriction to limit models ordered by universal extensions. Write ${\bf K}^*={\bf K}_{\mathfrak{i},\mathfrak{i}^*,\mathcal{P}(n)}^{\text{proper},*}$.
\begin{enumerate}
\item \cite[Lemma 10.15(5)]{ss2} Let $({\bf m}^0,{\bf m}^1,{\bf m}^2)$ be a $\phi_n$-span in ${\bf K}^*$ and write ${\bf m}^i=\langle M_u^i:u\in\mathcal{P}(n)\rangle$ for $i=0,1,2$. Then $({\bf m}^0,{\bf m}^1,{\bf m}^2)$ has a $\phi_n$-amalgam in ${\bf K}^*$ iff there exists $N\in K$, $f_i:M^i_{\mathcal{P}(n)}\xrightarrow[{\bf m}^0]{}N$ for $i=1,2$ such that $f_1\restriction M_u^1=f_2\restriction M_u^2$ for $u\in\mathcal{P}(n)\backslash\{n\}$.
\item \cite[Lemma 11.16(2)]{ss2} Let $\lambda,\lambda^+$ be in the domain of $\mathfrak{i}$. Suppose $2^\lambda<2^{\lambda^+}$ and for $\mu=\lambda,\lambda^+$, $\mathfrak{i}^*$ has $(\mu,n)$-existence and $(\mu,n)$-uniqueness. Then $\mathfrak{i}^*$ also has $(\lambda,n+1)$-uniqueness.
\end{enumerate}
\end{fact}
\begin{corollary}\mylabel{stepcor}{Corollary \thetheorem}
Let $n<\omega$, $\mathfrak{i}$ be a very good multidimensional independence relation defined on $\mathcal{P}(n+1)$-systems, $\mathfrak{i}^*$ be its restriction to limit models ordered by universal extensions. Let $\lambda,\lambda^+$ be in the domain of $\mathfrak{i}$. Suppose ${\bf K}$ has amalgamation over sets and for $\mu=\lambda,\lambda^+$, $\mathfrak{i}^*$ has $(\mu,n)$-existence and $(\mu,n)$-uniqueness. Then $\mathfrak{i}^*$ also has $(\lambda,n+1)$-uniqueness.
\end{corollary}
\begin{proof}
WGCH was used in the proof of \ref{stepfact}(2) to show that there is a $\phi_n$-amalgamation base $K^*_\lambda$. It suffices to show that the second part of \ref{stepfact}(1) is always true under amalgamation over sets, which will imply that \emph{any} ${\bf m}^0$ is a $\phi_n$-amalgamation base. \\ Let $({\bf m}^0,{\bf m}^1,{\bf m}^2)$ as in \ref{stepfact}(1). We observe the following:
\begin{enumerate}
\item The models in ${\bf m}^0$ are ${\bf K}$-substructures of $M_{\mathcal{P}(n)}^0\leq M_{\mathcal{P}(n)}^i$ for $i=1,2$. In particular ${\bf m}^0$ is a common subset of the latter two.
\item Since $({\bf m}^0,{\bf m}^1,{\bf m}^2)$ is a $\phi_n$-span, ${\bf m}^1$ and ${\bf m}^2$ agree on $\phi_n$, which means that for $u\in\mathcal{P}(n)\backslash\{n\}$, $M_u^1=M_u^2$.
\end{enumerate} 
Now take $A$ be the union of the models in ${\bf m}^0$ as well as $M_u^i$ for $u\in\mathcal{P}(n)\backslash\{n\}$, $i=1,2$. Then we can invoke amalgamation over sets to obtain $f_i:M_{\mathcal{P}(n)}^i\xrightarrow[A]{} N$ for some $N\in K$. By (2), $f_1\restriction M_u^1=\oop{id}=f_2\restriction M_u^2$.
\end{proof}
\begin{remark}
In the above proof, we can relax amalgamation over sets to amalgamation over multiple models. Namely, let $\langle M_u:u\in I\rangle$ be a finite set of models in ${\bf K}$. Suppose each $M_u$ is a ${\bf K}$-substructure  of $N_1$ and $N_2$, then there are $N\in K$ and $f_i:N_i\xrightarrow[\bigcup_{u\in I}M_u]{}N$ for $i=1,2$. The point in the original proof of \ref{stepfact}(2) is to restrict the class to a nice enough ${\bf K}^*$ so that WGCH is sufficient.
\end{remark}
We will show that the frame in \ref{extfact} guarantees that the AEC of sufficiently saturated models is excellent, has primes and hence allows categoricity transfer. 
\begin{definition}\cite[Definition 13.1]{ss2}
\begin{enumerate} 
\item Let $\mathfrak{i}$ be a multidimensional independence relation. $\mathfrak{i}$ is \emph{excellent} if 
\begin{enumerate}
\item $\mathfrak{i}$ is defined on an AEC ${\bf K}^*$;
\item $\mathfrak{i}$ is very good \cite[Definition 11.2]{ss2};
\item $\mathfrak{i}$ has extension and uniqueness \cite[Definitions 8.11, 8.16]{ss2}.
\end{enumerate}
\item An AEC ${\bf K^*}$ is \emph{excellent} if there is an excellent multidimensional independence relation defined on ${\bf K^*}$.
\end{enumerate}
\end{definition}
\begin{fact}\mylabel{excfact}{Fact \thetheorem} 
\begin{enumerate}
\item Let ${\bf K}$ be an AEC. Suppose there is a $(<\omega)$-extendible categorical very good $\lambda$-frame $\mathfrak{s}$ defined on some ${\bf K}_{\mathfrak{s}}$. Let ${\bf K}^*$ be the AEC generated by ${\bf K}_{\mathfrak{s}}$. If WGCH holds, then ${\bf K}^*$ is excellent.
\item Let ${\bf K}^*$ be an AEC. If ${\bf K^*}$ is excellent, then ${\bf K}^{{\oop{LS}({\bf K}^*)}^+\text{-sat}}$ has primes.
\end{enumerate}
\end{fact}
\begin{corollary}\mylabel{lastcor}{Corollary \thetheorem}
Let ${\bf K}$ be a $\ls$-short AEC with amalgamation over sets and arbitrarily large models. Suppose ${\bf K}$ is superstable in $\ls$ and let $\lambda=(2^{\ls})^{+6}$, then ${\bf K}^{\lambda\text{-sat}}$ is excellent and has primes.
\end{corollary}
\begin{proof}
Let $\lambda^-$ be the predecessor cardinal of $\lambda$. By \ref{extfact}, there is a $(<\omega)$-extendible categorical very good $(\geq\lambda^-)$-frame $\mathfrak{s}$ defined on ${\bf K}^*\defeq K^{\lambda^-\text{-sat}}$ (which is also the AEC generated by ${\bf K}_{\lambda^-}^{\lambda^-\text{-sat}}$; see \ref{s3fact}(1)). In the proof of \ref{excfact}(1), the only usage of WGCH is to show \ref{stepfact}(2), which can be replaced by amalgamation over sets due to \ref{stepcor}. Hence ${\bf K}^*$ is excellent. By \ref{excfact}(2), ${\bf K}^{\lambda\text{-sat}}$ has primes. Restart the whole proof with $\lambda^-$ replaced by $\lambda$ to obtain excellence for ${\bf K}^{\lambda\text{-sat}}$.
\end{proof}
\begin{remark}
Excellence (a nice enough multidimensional independence relation) is an important tool to generalize the main gap theorem to uncountable theories. \cite[Section 1.3]{ss2} already hinted that their result (with non-ZFC assumptions) satisfies (part of) the axioms of \cite{GLmain}. Here we obtain a ZFC version of excellence by assuming amalgamation over sets. This is perhaps not a strong assumption because we still do not have a proof of the main gap theorem for uncountable first-order theories. Future work in this direction could be verifying \cite[Axioms 8-10]{GLmain} on regular types. Relevant results can be found in \cite[III]{shh} but the definitions are different from those in \cite{GLmain}.
\end{remark}
We state two last facts before proving the categoricity transfer in the abstract. The proof of the first fact uses orthogonality calculus while the proof of the second fact uses Shelah's omitting type theorem in \cite{ms90} (see also \cite{bsott}).
\begin{fact}
\begin{enumerate}
\item \cite[Theorem 0.1]{s14}\mylabel{lastfact}{Fact \thetheorem}
Let ${\bf K}$ be an AEC and $\ls\leq\lambda<\theta$. Suppose ${\bf K}$ has a (type-full) good $[\lambda,\theta]$-frame and is categorical in $\lambda,\theta^+$, then it is categorical in all $\mu\in[\lambda,\theta]$. 
\item \cite[Theorem 9.8]{s14} Let ${\bf K}$ be an $\ls$-tame AEC with arbitrarily large models. If it is categorical in some $\lambda>\ls$, then the categoricity spectrum contains $h(\ls)$ and is unbounded.
\end{enumerate}
\end{fact}
\begin{theorem}\mylabel{lastthm}{Theorem \thetheorem}
Let ${\bf K}$ be an AEC which is $\ls$-short and has a monster model (with amalgamation over sets). Suppose ${\bf K}$ is categorical in some $\xi>\ls$, then it is categorical in all $\xi'\geq\min(\xi,h(\ls))$.
\end{theorem}
\begin{proof}
We follow the same idea in \cite{s14,ss2}, where we obtain primes for sufficiently saturated models by the results in this section, then transfer categoricity by Section \ref{apsec5} and the above fact. Categoricity also bootstraps the original AEC to be eventually categorical.
\begin{enumerate}
\item By \ref{catfat}, ${\bf K}$ is superstable in $\ls$. Let $\lambda=(2^{\ls})^{+6}$. By \ref{lastcor}, ${\bf K}^*\defeq{\bf K}^{\lambda\text{-sat}}$ has primes.
\item By \ref{lastfact}(2), we may assume that ${\bf K}$ (hence ${\bf K}^*$) is categorical in some $\theta>\lambda=\oop{LS}({\bf K}^*)$. By \ref{primefact}, ${\bf K}^*$ is categorical in all $\lambda'\geq\min(\theta,h(\oop{LS}({\bf K}^*)))=\min(\theta,h(\ls))$. In particular it is categroical in $\theta^+$.
\item Since ${\bf K}^*$ is categorical in $\lambda$ (by saturation) and $\theta^+$, by \ref{lastfact}(1) it is categorical in all $\lambda'\in[\lambda,\theta]$. Combining with (2), it is categorical in all $\lambda'\geq\lambda$. 
\item By \ref{lastfact}(2), we may assume $\xi\leq h(\ls)$. We consider two cases:
\begin{enumerate}
\item $\xi\geq\lambda$: the models in ${\bf K}_\xi$ are saturated, in particular $\lambda$-saturated. Hence ${\bf K}_{\geq\xi}={\bf K}^*_{\geq\xi}$ is totally categorical as desired.
\item $\xi<\lambda$: by \ref{s3fact}(2), there is a good $(\geq\xi)$-frame over ${\bf K}^{**}\defeq{\bf K}^{\xi\text{-sat}}$. ${\bf K}^{**}$ is categorical in $\xi$ by saturation. By substituting $\xi$ by $h(\ls)$ in (a), we have ${\bf K}$, and hence ${\bf K}^{**}$ is categorical in all $\xi'\geq h(\ls)$. In particular ${\bf K}^{**}$ is categorical in $h(\ls)^+$. By the same argument as (3), ${\bf K}^{**}$ is categorical in all $\xi'\geq\xi$. Now we end up in the scenario of (a) with the new ``$\lambda$'' being $\xi$ so ${\bf K}_{\geq\xi}={\bf K}^{**}_{\geq\xi}$ which is totally categorical.
\end{enumerate}
\end{enumerate}
\end{proof}
We apply our theorem to prove known results:
\begin{example}
\begin{enumerate}
\item Complete first-order theories: by compactness the models of a complete first-order theory $T$ satisfy amalgamation over sets, joint-embedding and no maximal models. It has L\"{o}wenheim-Skolem number $|T|$ and is $(<\al)$-short. Therefore, we can use \ref{lastthm} transfer categoricity in any $\mu>|T|$ to all $\mu'\geq\mu$. However, we cannot conclude categoricity down to all $\mu'>|T|$ as in \cite{morcat,sh31} which used syntactic proofs.
\item Homogeneous diagrams with a monster model: let $T$ be a first-order theory and $D$ be a subset of syntactic $T$-types over the empty set. Let ${\bf K}_D$ be the class of models of $T$ such that the only types over the empty set they realize are from $D$, where the models are ordered by elementary substructures. Assuming the existence of a monster model (see the precise statements in \cite[Hypothesis 2.5]{GLspec} or \cite[Definition 4.2]{s11}), we have the same properties as those in (1). Hence we can transfer categoricity in any $\mu>|T|$ to all $\mu'\geq\mu$. \cite[Theorem 4.22]{s11} proved the same result using \ref{primefact} but also syntactic results from \cite{sh3}. Our approach is purely semantic. 
\end{enumerate}
\end{example}
Our theorem does not exclude the possibility that the first categoricity cardinal to be arbitrarily close to $h(\ls)$. The following example shows such categoricity behavior but unfortunately it fails amalgamation and joint-embedding.
\begin{example}\mylabel{lastexam}{Example \thetheorem}
Let $\lambda\geq\al$ and $\lambda\leq\alpha<(2^\lambda)^+$. By the construction of ${\bf K}_0$ and ${\bf K}_1$ in \cite[Proposition 4.1]{leung1}, there is ${\bf K}^\alpha$, an AEC that encodes the cumulative hierarchy $V_\alpha(\alpha)$. ${\bf K}^\alpha$ is ordered by $\oop{L}({\bf K}^\alpha)$-substructures, $\oop{LS}({\bf K}^\alpha)=\lambda$ and the models have sizes up to $\beth_\alpha(\lambda)$. Also, ${\bf K^\alpha}$ has joint-embedding but not amalgamation. Taking the disjoint union of ${\bf K}^\alpha$ with a totally categorical AEC, we obtain an AEC ${\bf K}$ whose first categoricity cardinal is $\beth_{\alpha}(\lambda)$, but it fails amalgamation and joint-embedding. 
\end{example}
\begin{remark}
\cite[Example 9.10(2)]{s31} claimed that by encoding the cumulative hierarchy, one could get such an example with amalgamation (which would provide a complete list of examples for his categoricity spectra). However, he did not provide the exact encoding or the ordering (which amalgamation is sensitive to), so we cannot verify his claim. A similar problem occurs in \cite[Example 9.10(3)]{s31} when he encoded an AEC ${\bf K}$ categorical only in $[{\ls}^{+m},{\ls}^{+n}]$ where $m,n<\omega$. If we use $\llk$-substructures as the ordering, amalgamation again fails because the functions (see $F$ in \cite[Fact 9.8]{s31}) might be computed differently.
\end{remark}
\begin{ques}
Let $\kappa\geq\al$. For $\mu<h(\kappa)$, is there an AEC ${\bf K}$ with $\ls=\kappa$ which is $\kappa$-short, has amalgamation over sets and arbitrarily large models such that the first categoricity cardinal (exists and) is greater than $\mu$? What if we replace amalgamation over sets by the usual amalgamation property (see also \hyperlink{catlist}{table}(10)(11))? 
\end{ques}
\bibliographystyle{alpha}
\bibliography{references}

{\small\setlength{\parindent}{0pt}
\textit{Email}: wangchil@andrew.cmu.edu

\textit{URL}: http://www.math.cmu.edu/$\sim$wangchil/

\textit{Address}: {Department of Mathematical Sciences, Carnegie Mellon University, Pittsburgh PA 15213, USA}
\end{document}